\documentclass[a4paper,12pt,final]{amsart}
\usepackage{times,a4wide,mathrsfs,amssymb,amsmath,amsthm,wasysym}

\newcommand{\C}{\mathbb{C}}

\newcommand{\LLL}{\mathbb{L}}
\newcommand{\QQ}{\mathbb{Q}}
\newcommand{\NN}{\mathbb{N}}
\newcommand{\PP}{\mathbb{P}}

\newcommand{\Cc}{\mathcal C}

\newcommand{\MM}{\mathcal M}

\newcommand{\gr}{\hbox{Gr}}
\newcommand{\wt}{\widetilde}
\newcommand{\rom}{\romannumeral}

 \makeatletter
\newcommand*{\da@rightarrow}{\mathchar"0\hexnumber@\symAMSa 4B }
\newcommand*{\da@leftarrow}{\mathchar"0\hexnumber@\symAMSa 4C }
\newcommand*{\xdashrightarrow}[2][]{%
  \mathrel{%
    \mathpalette{\da@xarrow{#1}{#2}{}\da@rightarrow{\,}{}}{}%
  }%
}
\newcommand{\xdashleftarrow}[2][]{%
  \mathrel{%
    \mathpalette{\da@xarrow{#1}{#2}\da@leftarrow{}{}{\,}}{}%
  }%
}
\newcommand*{\da@xarrow}[7]{%
  \sbox0{$\ifx#7\scriptstyle\scriptscriptstyle\else\scriptstyle\fi#5#1#6\m@th$}%
  \sbox2{$\ifx#7\scriptstyle\scriptscriptstyle\else\scriptstyle\fi#5#2#6\m@th$}%
  \sbox4{$#7\dabar@\m@th$}%
  \dimen@=\wd0 %
  \ifdim\wd2 >\dimen@
    \dimen@=\wd2 %   
  \fi
  \count@=2 %
  \def\da@bars{\dabar@\dabar@}%
  \@whiledim\count@\wd4<\dimen@\do{%
    \advance\count@\@ne
    \expandafter\def\expandafter\da@bars\expandafter{%
      \da@bars
      \dabar@ 
    }%
  }%  
  \mathrel{#3}%
  \mathrel{%   
    \mathop{\da@bars}\limits
    \ifx\\#1\\%
    \else
      _{\copy0}%
    \fi
    \ifx\\#2\\%
    \else
      ^{\copy2}%
    \fi
  }%   
  \mathrel{#4}%
}
\makeatother

\DeclareMathOperator{\aut}{Aut}

\DeclareMathOperator{\ide}{id}

\DeclareMathOperator{\ima}{Im}

\newtheorem{theorem}{Theorem}[section]

\newtheorem{lemma}[theorem]{Lemma}

\newtheorem{proposition}[theorem]{Proposition}
\newtheorem{conjecture}[theorem]{Conjecture}
\newtheorem{remark}[theorem]{Remark}
\newtheorem{definition}[theorem]{Definition}
\newtheorem{convention}{Conventions}

\newtheorem{notation}[theorem]{Notation}

\newtheorem{nonumbering}{Theorem}

\newtheorem{nonumberingt}{Acknowledgements}

\begin{document}
\author[Robert Laterveer]
{Robert Laterveer}

\address{Institut de Recherche Math\'ematique Avanc\'ee,
CNRS -- Universit\'e 
de Strasbourg,\
7 Rue Ren\'e Des\-car\-tes, 67084 Strasbourg CEDEX,
FRANCE.}
\email{robert.laterveer@math.unistra.fr}

\title{Algebraic cycles and very special cubic fourfolds}

\begin{abstract} Informed by the Bloch--Beilinson conjectures, Voisin has made a conjecture about $0$--cycles on self--products of Calabi--Yau varieties. In this note, we consider variant versions of Voisin's conjecture for cubic fourfolds, and for hyperk\"ahler varieties. We present examples for which these conjectures are verified, by considering certain very special cubic fourfolds and their Fano varieties of lines.
 \end{abstract}


\maketitle

\section{Introduction}

Let $X$ be a smooth projective variety over $\C$, and let $A^i(X):=CH^i(X)_{\QQ}$ denote the Chow groups of $X$ (i.e. the groups of codimension $i$ algebraic cycles on $X$ with $\QQ$--coefficients, modulo rational equivalence). With respect to any reasonable topology, the world of algebraic cycles is densely filled with open problems \cite{B}, \cite{J2}, \cite{MNP}, \cite{Vo}. One of these open problems is the following intriguing conjecture of Voisin, which can be seen as a version of Bloch's conjecture for varieties of geometric genus one:

\begin{conjecture}[Voisin \cite{V9}]\label{conjvois} Let $X$ be a smooth projective complex variety of dimension $n$ with $h^{j,0}(X)=0$ for $0<j<n$ and $p_g(X)=1$.  
For any $0$--cycles $a,a^\prime\in A^n(X)$ of degree zero, we have
  \[ a\times a^\prime=(-1)^n a^\prime\times a\ \ \ \hbox{in}\ A^{2n}(X\times X)\ .\]
  (Here $a\times a^\prime$ is short--hand for the cycle class $(p_1)^\ast(a)\cdot(p_2)^\ast(a^\prime)\in A^{2n}(X\times X)$, where $p_1, p_2$ denote projection on the first, resp. second factor.)
   \end{conjecture}
   
 Conjecture \ref{conjvois} is still wide open for a general $K3$ surface (on the positive side, cf. \cite{V9}, \cite{moi}, \cite{des}, \cite{tod}, \cite{BLP} for some cases where this conjecture is verified).

Let us now suppose that $X$ is a hyperk\"ahler variety (i.e., a projective irreducible holomorphic symplectic manifold, cf. \cite{Beau0}, \cite{Beau1}). Conjecture \ref{conjvois} does not apply verbatim to $X$ (the Calabi--Yau condition 
$h^{j,0}(X)=0$ for $0<j<n$ is not satisfied), yet one can adapt conjecture \ref{conjvois} to make sense for $X$. For this adaptation, we will optimistically assume the Chow ring of $X$ has a bigraded ring structure $A^\ast_{(\ast)}(X)$, where each $A^i(X)$ splits into pieces
  \[ A^i(X) =\bigoplus_j A^i_{(j)}(X)\ .\]
 (Conjecturally, such a splitting exists for all hyperk\"ahler varieties, and the piece $A^i_{(j)}(X)$ should be isomorphic to the graded $\gr^j_F A^i(X)$ for the conjectural Bloch--Beilinson filtration \cite{Beau3}.) 
 Since the piece $A^i_{(j)}(X)$ should be related to the cohomology group $H^{2i-j}(X)$, and 
   \[ \wedge^2  H^{s}(X)\ \subset\ H^{2s}(X\times X) \]
   should be supported on a divisor for any $s$ (in view of the generalized Hodge conjecture), we arrive at the following version of conjecture \ref{conjvois}:

\begin{conjecture}\label{conjHK} Let $X$ be a hyperk\"ahler variety of dimension $2m$. Let $a,a^\prime\in A^{2m}_{(j)}(X)$. Then
  \[ a\times a^\prime -a^\prime\times a=0\ \ \ \hbox{in}\ A^{4m}_{}(X\times X)\ .\]
  \end{conjecture}
  
  (We note that in conjecture \ref{conjHK}, we silently presuppose that $A^{2m}_{(j)}(X)=0$ for $j$ odd; this is the case if the bigrading is related to the conjectural Bloch--Beilinson filtration.) Conjecture \ref{conjHK} is verified for a family of hyperk\"ahler fourfolds in \cite{voishk}.
  
 We can also consider a variant of conjecture \ref{conjHK} for cubic fourfolds (the point being that cubic fourfolds have $h^{4,0}=0$ and $h^{3,1}=1$, so cohomologically they look like a ``shifted Calabi--Yau variety''):
 
 \begin{conjecture}\label{conjcube} Let $Y\subset\PP^5(\C)$ be a smooth cubic fourfold. Let $a,a^\prime\in A^3_{hom}(X)$. Then
    \[ a\times a^\prime -a^\prime\times a=0\ \ \ \hbox{in}\ A^{6}_{}(Y\times Y)\ .\]
  \end{conjecture}  
  
 Thanks to work of Shen--Vial \cite{SV}, the truth of conjecture \ref{conjcube} for a given smooth cubic $Y$ is equivalent to the truth of conjecture \ref{conjHK} for the Fano variety of lines on $Y$, cf. proposition \ref{equiv}.

  The main result of this note is that these conjectures are true in certain special cases:

 \begin{nonumbering}[=theorem \ref{main}]  Let $Y\subset\PP^5(\C)$ be a very special cubic. Then conjecture \ref{conjcube} is true for $Y$.
 
  Consequently, conjecture \ref{conjHK} is true for the Fano variety $X=F(Y)$ of lines in $Y$ (here, the notation $A^\ast_{(\ast)}(X)$ refers to the Fourier decomposition of \cite{SV}).
  \end{nonumbering}  
  
 Very special cubics (cf. definition \ref{def}) have the following property: there exist a countably infinite number of divisors $\Cc_d$ in the moduli space of cubic fourfolds (parametrizing special cubics in the sense of \cite{Has}), such that the very special cubics lie analytically dense in each $\Cc_d$.

  We also exhibit two explicit cubics for which these conjectures hold:
  
 \begin{nonumbering}[=theorem \ref{main2}] Let $Y\subset\PP^5(\C)$ be either the Fermat cubic
   \[ X_0^3 + X_1^3 +\cdots + X_5^3=0\ ,\]
or the smooth cubic defined as
   \[  X_0^3 + X_1^2X_5 + X_2^2X_4 + X_3^2X_2 +X_4^2X_1 +X_5^2X_3=0\ .\] 
 Then conjecture \ref{conjcube} holds for $Y$, and conjecture \ref{conjHK} holds for the Fano variety $X=F(Y)$.
  \end{nonumbering}

 \vskip0.6cm

\begin{convention} In this article, the word {\sl variety\/} will refer to a reduced irreducible scheme of finite type over $\C$. A {\sl subvariety\/} is a (possibly reducible) reduced subscheme which is equidimensional. 

{\bf All Chow groups will be with rational coefficients}: we will denote by $A_j(X)$ the Chow group of $j$--dimensional cycles on $X$ with $\QQ$--coefficients; for $X$ smooth of dimension $n$ the notations $A_j(X)$ and $A^{n-j}(X)$ are used interchangeably. 

The notations $A^j_{hom}(X)$, $A^j_{AJ}(X)$ will be used to indicate the subgroups of homologically trivial, resp. Abel--Jacobi trivial cycles.
%For a morphism $f\colon X\to Y$, we will write $\Gamma_f\in A_\ast(X\times Y)$ for the graph of $f$.
%The contravariant category of Chow motives (i.e., pure motives with respect to rational equivalence as in \cite{Sc}, \cite{MNP}) will be denoted $\MM_{\rm rat}$.

%The Griffiths group $\grif^j$ is the group of codimension $j$ cycles that are homologically trivial modulo algebraic equivalence, again with $\QQ$--coefficients. 

We use $H^j(X)$ 
to indicate singular cohomology $H^j(X,\QQ)$.
\end{convention}

\section{Preliminaries}

\subsection{Finite--dimensional motives}

We refer to \cite{Kim}, \cite{An}, \cite{J4}, \cite{MNP} for the definition of finite--dimensional motive. 
An essential property of varieties with finite--dimensional motive is embodied by the nilpotence theorem:

\begin{theorem}[Kimura \cite{Kim}]\label{nilp} Let $X$ be a smooth projective variety of dimension $n$ with finite--dimensional motive. Let $\Gamma\in A^n(X\times X)_{}$ be a correspondence which is numerically trivial. Then there is $N\in\NN$ such that
     \[ \Gamma^{\circ N}=0\ \ \ \ \in A^n(X\times X)_{}\ .\]
\end{theorem}

 Actually, the nilpotence property (for all powers of $X$) could serve as an alternative definition of finite--dimensional motive, as shown by Jannsen \cite[Corollary 3.9]{J4}.
Conjecturally, any variety has finite--dimensional motive \cite{Kim}. We are still far from knowing this, but at least there are quite a few non--trivial examples.

\subsection{CK decomposition}

\begin{definition}[Murre \cite{Mur}] Let $X$ be a smooth projective variety of dimension $n$. We say that $X$ has a {\em CK decomposition\/} if there exists a decomposition of the diagonal
   \[ \Delta_X= \pi_0+ \pi_1+\cdots +\pi_{2n}\ \ \ \hbox{in}\ A^n(X\times X)\ ,\]
  such that the $\pi_i$ are mutually orthogonal idempotents and $(\pi_i)_\ast H^\ast(X)= H^i(X)$.
  
  (NB: ``CK decomposition'' is shorthand for ``Chow--K\"unneth decomposition''.)
\end{definition}

\begin{remark} The existence of a CK decomposition for any smooth projective variety is part of Murre's conjectures \cite{Mur}, \cite{J2}. 
%If a quotient variety $X$
%has finite--dimensional motive, and the K\"unneth components are algebraic, then $X$ has a CK decomposition (this can be proven just as \cite{J2}, where this is stated for smooth $X$).
\end{remark}

In what follows, we will make use of the following: 

\begin{theorem}[Shen--Vial \cite{SV}]\label{fanomck} Let $Y\subset\PP^5(\C)$ be a smooth cubic fourfold, and let $X:=F(Y)$ be the Fano variety of lines in $Y$. There exists a CK decomposition $\{\pi^X_i\}$ for $X$, and 
  \[ (\pi^X_{2i-j})_\ast A^i(X) = A^i_{(j)}(X)\ ,\]
  where the right--hand side denotes the splitting of the Chow groups defined in terms of the Fourier transform as in \cite[Theorem 2]{SV}. Moreover, we have
  \[ A^i_{(j)}(X)=0\ \ \ \hbox{for\ }j<0\ \hbox{and\ for\ }j>i\ .\]
  
  In case $Y$ is very general, the Fourier decomposition $A^\ast_{(\ast)}(X)$ forms a bigraded ring.    
  \end{theorem}

\begin{proof} (A remark on notation: what we denote $A^i_{(j)}(X)$ is denoted $CH^i(X)_j$ in \cite{SV}.)

The existence of a CK decomposition $\{\pi^X_i\}$ is \cite[Theorem 3.3]{SV}, combined with the results in \cite[Section 3]{SV} to ensure that the hypotheses of \cite[Theorem 3.3]{SV} are satisfied. According to \cite[Theorem 3.3]{SV}, the given CK decomposition agrees with the Fourier decomposition of the Chow groups. The ``moreover'' part is because the $\{\pi^X_i\}$ are shown to satisfy Murre's conjecture B \cite[Theorem 3.3]{SV}.

The statement for very general cubics is \cite[Theorem 3]{SV}.
    \end{proof}

\begin{remark}\label{pity} Unfortunately, it is not yet known whether the Fourier decomposition of \cite{SV} induces a bigraded ring structure on the Chow ring for {\em all\/} Fano varieties of smooth cubic fourfolds. (That is, it is not known whether the CK decomposition of theorem \ref{fanomck} is a {\em weak multiplicative\/} CK decomposition, in the sense of \cite{SV}.) For one thing, it has not yet been proven that $A^2_{(0)}(X)\cdot A^2_{(0)}(X)\subset A^4_{(0)}(X)$ (cf. \cite[Section 22.3]{SV} for discussion).
\end{remark}

\subsection{Multiplicative structure}

Let $X$ be the Fano variety of lines on a smooth cubic fourfold. As we have seen (theorem \ref{fanomck}), the Chow ring of $X$ splits into pieces $A^i_{(j)}(X)$.
The magnum opus \cite{SV} contains a detailed analysis of the multiplicative behaviour of these pieces. Here are the relevant results we will be needing:

\begin{theorem}[Shen--Vial \cite{SV}]\label{chowringfano} Let $Y\subset\PP^5(\C)$ be a smooth cubic fourfold, and let $X:=F(Y)$ be the Fano variety of lines in $Y$. 

\noindent
(\rom1) There exists $\ell\in A^2_{(0)}(X)$ such that intersecting with $\ell$ induces an isomorphism
  \[ \cdot\ell\colon\ \ \ A^2_{(2)}(X)\ \xrightarrow{\cong}\ A^4_{(2)}(X)\ .\]

\noindent
(\rom2) Intersection product induces a surjection
  \[ A^2_{(2)}(X)\otimes A^2_{(2)}(X)\ \twoheadrightarrow\ A^4_{(4)}(X)\ .\]
\end{theorem} 
     
 \begin{proof} Statement (\rom1) is \cite[Theorem 4]{SV}. Statement (\rom2) is \cite[Proposition 20.3]{SV}.
  \end{proof}

\subsection{The two conjectures are equivalent}

\begin{proposition}\label{equiv} Let $Y\subset \PP^5(\C)$ be a smooth cubic fourfold, and let $X=F(Y)$ be the Fano variety of lines in $Y$. Conjecture \ref{conjcube} holds for $Y$ if and only if conjecture \ref{conjHK} holds for $X$.
\end{proposition}

\begin{proof} Let
  \[ \begin{array}[c]{ccc}
       P &\to& Y\\
       \downarrow&&\\
       X&&\\
       \end{array}\]
       denote the universal family of lines. For a point $s\in X$, let $\ell_s\subset Y$ be the corresponding line. Let
       \[ I := \bigl\{  (s,t)\in X\times X\ \vert\ \ell_s\cap \ell_t\not=\emptyset\bigr\}  \]
       denote the incidence correspondence. Viewing $P$ as a correspondence $P\in A^3(X\times Y)$, there is the relation
       \begin{equation}\label{inci} I={}^t P\circ P\ \ \ \hbox{in}\ A^2(X\times X) \end{equation}
       \cite[Lemma 17.2]{SV}.
       
Let us suppose conjecture \ref{conjcube} holds for $Y$. We observe that there is equality
  \[ A^2_{(2)}(X)= I_\ast A^4_{hom}(X) \]
  \cite[Proof of Proposition 21.10]{SV}. Using the relation (\ref{inci}), this means that
  \[  A^2_{(2)}(X)= ({}^t P)_\ast  P_\ast A^4_{hom}(X)\ . \]
  But $P_\ast\colon A^4_{hom}(X)\to A^3_{hom}(Y)$ is known to be surjective \cite{Par}, and so
    \[  A^2_{(2)}(X) = ({}^t P)_\ast A^3_{hom}(Y)\ .\] 
Let $b, b^\prime\in A^2_{(2)}(X)$. We can write $b=({}^t P)_\ast(a)$ and $b^\prime=({}^t P)_\ast(a)$, for some $a, a^\prime\in A^3_{hom}(Y)$.
But then
  \[ b\times b^\prime - b^\prime\times b = ({}^t P\times {}^t P)_\ast (a\times a^\prime - a^\prime\times a) =0\ \ \ \hbox{in}\ A^4(X\times X)\ .\]
  Using theorem \ref{chowringfano}(\rom1), this implies conjecture \ref{conjHK} is true for $A^4_{(2)}(X)$:
  given $a,a^\prime\in A^4_{(2)}(X)$, there exist $b,b^\prime\in A^2_{(2)}(X)$ such that $a=b\cdot\ell$ and $a^\prime=b^\prime\cdot\ell$. Thus, we find
  \[  a\times a^\prime = b\cdot\ell\times b^\prime\cdot\ell = (b\times b^\prime)\cdot (\ell\times\ell) = (b^\prime\times b)\cdot(\ell\times\ell) = a^\prime\times a\ \ \ \hbox{in}\ A^8(X\times X)\ .\]
   Using theorem \ref{chowringfano}(\rom2), one checks conjecture \ref{conjHK} is also true for $A^4_{(4)}(X)$.
  
  Next, let us suppose conjecture \ref{conjHK} holds for $X$. As mentioned above, there is a surjection
  \[ P_\ast\colon \ \ \ A^4_{hom}(X)\to A^3_{hom}(Y)  \]
  \cite{Par}. Moreover, one has $A^4_{hom}(X)=A^4_{(2)}(X)\oplus A^4_{(4)}(X)$. Since 
   \[ A^4_{(4)}(X)= \ker\bigl( A^4(X)\ \xrightarrow{P_\ast}\ A^3(Y) \]
   \cite[Theorem 20.5]{SV}, the map
   \[ P_\ast\colon \ \ \ A^4_{(2)}(X)\to A^3_{hom}(Y)  \]
   is still a surjection. Hence, one can deduce conjecture \ref{conjcube} for $Y$ from the truth of conjecture \ref{conjHK} for $A^4_{(2)}(X)$.
\end{proof}

\subsection{Very special cubics}

\begin{definition}\label{def} Let $Y\subset\PP^5(\C)$ be a smooth cubic fourfold. We say that $Y$ is {\em very special\/} if

\noindent
(1) there exists a $K3$ surface $S$ such that the Fano variety $F(Y)$ (of lines contained in $Y$) is birational to the Hilbert scheme $S^{[2]}$, and

\noindent
(2) the dimension of $N^2 H^4(Y)$ is $\ge 20$ (or equivalently, the Picard number $\rho(S)$ is $\ge 19$).
\end{definition}

\begin{notation} As in \cite{Has} and \cite{Has2}, we will write $\Cc$ for the $20$--dimensional moduli space of smooth cubic fourfolds, and $\Cc_d$ for the divisor parametrizing special cubics admitting a labelling of discriminant $d$.
\end{notation}

The following shows there are quite many cubics that are very special:

\begin{theorem}[Addington \cite{Add}] Let $d$ be an integer of the form $d=2(n^2+n+1)/a^2$, where $n>1$ and $a>0$ are integers. Then the very special cubic fourfolds of discriminant $d$ form a union of curves that lies analytically dense in $\Cc_d$.
\end{theorem}

\begin{proof} Let $Y\in \Cc_d$.
Addington has proven \cite[Theorem 2]{Add} that there exists an associated $K3$ surface $S$ of degree $d$ such that
  \[ F(Y) \ \thicksim_{\rm birat}\ S^{[2]}\ ,\]
i.e. $Y$ satisfies condition (1) of definition \ref{def}. There are natural isomorphisms
  \[ H^4_{tr}(Y)\cong H^2_{tr}(F(Y))\cong H^2_{tr}(S^{[2]})\cong H^2_{tr}(S) \ .\]
  (The first follows from the Abel--Jacobi isomorphism established in \cite{BD}, the second is because $F(Y)$ and $S^{[2]}$ are birational, and the last follows from \cite[Proposition 6]{Beau1}.)
  This shows that the cubic $Y$ satisfies condition (2) of definition \ref{def} if and only if the associated $K3$ surface $S$ has Picard number $\ge 19$.

$K3$ surfaces of Picard number $\ge 19$ form a union of curves that lies analytically dense in the moduli space ${\mathcal N}_d$ of degree $d$ polarized $K3$ surfaces. The condition on $d$ implies that $d$ is {\em admissible\/}, in the sense of \cite[Definition 22]{Has2} (this means that $d$ satisfies condition (**) of \cite[Introduction]{Add}). It follows from Hassett's work \cite{Has}, \cite[Corollary 25]{Has2} that $\Cc_d$ is irreducible, and either birational to ${\mathcal N}_d$ or birational to the quotient of ${\mathcal N}_d$ under an involution. Either way, this implies that very special cubic fourfolds form a union of curves that lies analytically dense in $\Cc_d$.
\end{proof}

\begin{remark} The class of very special cubic fourfolds (as defined in definition \ref{def}) is less restrictive than the class of cubic fourfolds studied by Hulek--Kloosterman \cite[Corollaries 4.14 and 4.15]{HK}. (Indeed, in \cite{HK}, the authors ask for an isomorphism $F(Y)\cong S^{[2]}$ in (1), and for an equality $\dim N^2 H^4(Y)=21$ in (2).) The Hulek--Kloosterman cubics form discrete analytically dense subsets inside certain $\Cc_d$.
\end{remark}

\begin{remark} Condition (1) of definition \ref{def} is studied in \cite{GS}, where it is called {\em decomposability\/} of $F(Y)$. It is known that condition (1) is strictly more stringent than the condition of ``having an associated $K3$ surface'' in the sense of \cite{Has}, cf. \cite{Add} and \cite[Example 31]{Has2}.
\end{remark}

While we will not be needing this here, we mention in passing the following result:

\begin{proposition}\label{fd} A very special cubic has finite--dimensional motive.
\end{proposition}

\begin{proof} For any smooth cubic fourfold $Y\subset\PP^5(\C)$, Pedrini \cite[Section 4]{Ped2} has defined the ``transcendental part of the motive'' $t(Y)$ such that
  there is a decomposition
   \[ h(Y) \cong t(Y)\oplus \bigoplus \LLL(m_j)\ \ \ \hbox{in}\ \MM_{\rm rat}\ ,\]
   and such that $A^3_{hom}(Y)=A^\ast(t(Y))$.
   
   Moreover, in case the Fano variety of lines $F(Y)$ is birational to a Hilbert scheme $S^{[2]}$, there is an isomorphism
   \begin{equation}\label{chowiso} t(Y)\cong t(S)(1)\ \ \ \hbox{in}\ \MM_{\rm rat}\ ,\end{equation}
   where the right--hand side denotes the transcendental part of the motive of a surface \cite{KMP}.
   (NB: in \cite[Theorem 4.6]{Ped2}, the isomorphism (\ref{chowiso}) is established under the hypothesis that $F(Y)$ is isomorphic to $S^{[2]}$. However, in view of the fact that birational hyperk\"ahler varieties have isomorphic Chow motives \cite{Rie}, the same proof goes through when $F(Y)\thicksim_{\rm birat} S^{[2]}$.)
   
 Since any $K3$ surface with Picard number $\ge 19$ has finite--dimensional motive (cf. \cite{Ped}, or the proof of lemma \ref{19} below), the isomorphism (\ref{chowiso}) proves the proposition.
\end{proof}

\section{Main}

\begin{theorem}\label{main} Let $Y\subset\PP^5(\C)$ be a very special cubic. Then conjecture \ref{conjcube} holds for $Y$.
\end{theorem}

\begin{proof} Let $S$ be the $K3$ surface associated to $Y$ (so by definition, $S$ has Picard number $\ge 19$). We have already seen (in the course of the proof of proposition \ref{fd}) that there is an isomorphism of Chow motives
  \[  t(Y)\cong t(S)(1)\ \ \ \hbox{in}\ \MM_{\rm rat}\ .\]
  Taking Chow groups, this implies there is a correspondence--induced isomorphism
  \[ (\Gamma^\prime)_\ast\colon\ \ \ A^3_{hom}(Y)\ \xrightarrow{\cong}\ A^2_{hom}(S)\ .\]
 
 We need a little lemma:
  
 \begin{lemma}\label{19} Let $S$ be a $K3$ surface of Picard number $\rho(S)\ge 19$. Then there exist an abelian surface $A$, and a correspondence $\Psi\in A^2(S\times A)$ inducing an injection
   \[ \Psi_\ast\colon\ \ \ A^2_{hom}(S)\ \hookrightarrow\ A^2_{AJ}(A)\ .\]
 \end{lemma}
 
 \begin{proof} This is well--known; this is how finite--dimensionality of $S$ is proven. 
 
 First, let us assume $\rho(S)=20$. Then $S$ is either a Kummer surface, or there is a rational degree $2$ map $K\dashrightarrow S$ where $K$ is Kummer \cite{SI}. Clearly, this gives a correspondence $\Psi$ as in the lemma.
 
 Next, let us assume $\rho=19$. Then $S$ admits a Shioda--Inose structure, i.e. there exists an involution $i$ on $S$ such that the quotient $S/<i>$ is birational to a Kummer surface $K$ \cite{Mor}. The involution $i$ being symplectic, one has a correspondence--induced isomorphism
   \[ A^2(S)\ \xrightarrow{\cong}\ A^2(K) \]
 \cite{V11}. This induces $\Psi$ as in the lemma.
 \end{proof}  
 
 Using lemma \ref{19}, one reduces to a statement about $0$--cycles on abelian surfaces. Indeed, let $\Gamma$ be the correspondence
   \[ \Gamma:= \Psi\circ \Gamma^\prime\ \ \ \in A^3(Y\times A)\ .\]
 There is a commutative diagram
   \[ \begin{array}[c]{ccc}
      A^3_{hom}(Y)\otimes A^3_{hom}(Y) & \xrightarrow{\nu_Y} & A^6(Y\times Y)\\
      &&\\
       \downarrow{\scriptstyle (\Gamma_\ast, \Gamma_\ast)} && \downarrow{\scriptstyle (\Gamma\times\Gamma)_\ast}\\
       &&\\
      A^2_{AJ}(A)\otimes A^2_{AJ}(A) & \xrightarrow{\nu_A} & \ \, A^4(A\times A)\ ,\\ 
      \end{array} \]
      where $\nu_Y$ is defined as $\nu_Y(a,a^\prime):= a\times a^\prime - a^\prime\times a$, and $\nu_A$ is defined similarly.
      
   By construction, the left vertical arrow is injective. The right vertical arow is injective when restricted to $\ima \nu_Y$ (indeed, a left inverse is given by $(C\times C)_\ast$, where $C$ is a correspondence such that $C_\ast \Gamma_\ast=\ide\colon A^3_{hom}(Y)\to A^3_{hom}(Y)$). Hence, we are reduced to proving that
   $\nu_A$ is the zero map. This is a special case of the following more general result (combined with the fact that $A^2_{AJ}(A)$ coincides with the piece $A^2_{(2)}(A)$ of the Beauville splitting):
   
  \begin{proposition}[Voisin \cite{Vo}] Let $A$ be an abelian variety of dimension $g$. Let $b, b^\prime\in A^g_{(g)}(A)$, where $A^\ast_{(\ast)}(A)$ denotes the Beauville splitting of \cite{Beau}. There is equality
  \[ b\times b^\prime - (-1)^g \, b^\prime\times b=0\ \ \ \hbox{in}\ A^{2g}_{(2g)}(A\times A)\ .\]
  \end{proposition}
  
 \begin{proof} This is \cite[Example 4.40]{Vo}.
 \end{proof}
  
 The proof of theorem \ref{main} is now complete. 
\end{proof}

 \section{Two cubics}    

In this section, conjectures \ref{conjcube} and \ref{conjHK} are proven for two cubics:

   \begin{theorem}\label{main2} Let $Y\subset\PP^5(\C)$ be either the Fermat cubic
   \[ X_0^3 + X_1^3 + \cdots + X_5^3=0\ ,\]
or the smooth cubic defined as
   \[  X_0^3 + X_1^2X_5 + X_2^2X_4 + X_3^2X_2 +X_4^2X_1 +X_5^2X_3=0\ .\]
 Conjecture \ref{conjcube} is true for $Y$, i.e. for any two $1$--cycles $a,a^\prime\in A^3_{hom}(Y)$, there is equality
   \[ a\times a^\prime =  a^\prime\times a\ \ \ \hbox{in}\ A^6(Y\times Y)\ .\]
   Conjecture \ref{conjHK} is true for the Fano variety of lines $X=F(Y)$, i.e. for any two $0$--cycles $a,a^\prime\in A^4_{(j)}(X)$, there is equality
   \[ a\times a^\prime = a^\prime\times a\ \ \ \hbox{in}\ A^8(X\times X)\ .\]
   \end{theorem}  
     
I do not know whether the two cubics of theorem \ref{main2} are very special (if they are, the result would immediately follow from theorem \ref{main}). Therefore,
to prove theorem \ref{main2} we proceed slightly differently. Theorem \ref{main2} will be a consequence of the following result:

  \begin{theorem}\label{main1} Let $Y\subset\PP^5(\C)$ be a smooth cubic. Assume that $Y$ has finite--dimensional motive, and that
   \[ \dim N^2 H^4(Y)=21\ .\]
   Assume also that the Hodge conjecture is true for $Y\times Y$.
   Then conjecture \ref{conjcube} holds for $Y$, and conjecture \ref{conjHK} holds for the Fano variety $X=F(Y)$.
  \end{theorem}
  
%Let us see how this implies theorem \ref{main2}:

%\begin{proof}(of theorem \ref{main2}) We need to check the cubics of theorem \ref{main2} verify the two conditions of theorem \ref{main1}. Finite--dimensionality of the Fermat hypersurface is well--known; this follows from the Shioda inductive structure \cite{Shi}, \cite{KS}. Finite--dimensionality of the second cubic follows from \cite{excubic}, where it is proven that any smooth cubic fourfold of the form
%  \[ X_0^3 + f(X_1,\ldots, X_5)=0 \]
%  has finite--dimensional motive.
%  
%  The fact that $\dim N^2 H^4(Y)=21$ (i.e., $\dim H^4_{tr}(Y)=2$) is well--known for the Fermat cubic fourfold; Beauville \cite[Proposition 11]{Beau4} attributes this to Shioda. The fact that 
%    \begin{equation}\label{2} \dim H^4_{tr}(Y)=2\end{equation}
%  also for the second cubic is proven by Mongardi \cite{Mon}; this is an application of \cite[Proposition 1.2]{Mon} combined with the fact that the Fano variety $X=F(Y)$ admits an order $11$ symplectic automorphism.
%
%(Alternatively, a different proof of equality (\ref{2}) for the second cubic can be found in \cite[Section 5.5.2]{BNS}.)
%\end{proof} 

%It only remains to prove theorem \ref{main1}:

\begin{proof}(of theorem \ref{main1}) In view of proposition \ref{equiv}, it suffices to prove that conjecture \ref{conjcube} holds for $Y$.
We prove this by an argument similar to \cite[Theorem 4.1]{BLP}, which is the analogue of theorem \ref{main1} for Calabi--Yau varieties.

Let $\pi_{tr}\in A^4(Y\times Y)$ denote the idempotent defining the motive $t(Y)\in\MM_{\rm rat}$ of \cite{Ped2}. (Alternatively, one could define $\pi_{tr}:=\Pi_{3,1}$, where $\Pi_{i,j}$ refers to the refined Chow--K\"unneth decomposition of \cite[Theorems 1 and 2]{V4}.)
By construction, one has
  \[ \begin{split}  (\pi_{tr})_\ast H^\ast(Y) &= H^4_{tr}(Y)\ ,\\
                          (\pi_{tr})_\ast A^\ast(Y) &= A^3_{hom}(Y)\ .\\
                         \end{split}\]
        Let us consider the involution
         \[  \begin{split}  \iota\colon Y\times Y\ &\to Y\times Y\ , \\
          \iota(x,y)&:=(y,x)\ .\\
          \end{split}\]  
          We define a correspondence
          \[ \Gamma:= {1\over 2}  (\Delta_{Y\times Y}-\Gamma_\iota)\circ (\pi_{tr}\times\pi_{tr})\ \ \ \in A^8\bigl((Y\times Y)\times (Y\times Y)\bigr)\ .\]
      The correspondence $\Gamma$ is an idempotent (and actually, $\Gamma$ defines the motive $\wedge^2 t(Y)   \in \MM_{\rm rat}$ in the language of Kimura \cite{Kim}). 
     %   The correspondence $\Gamma$ acts on Chow groups as projector on the subgroup
    %  \[   \bigl\{ b\times b^\prime - b^\prime\times b\ \vert\ b,b^\prime\in A^3_{hom}(Y)\bigr\}\ \ \ \subset\ A^6(Y\times Y)\ .\]    (Not sure, but anyway don't need !!)
       The correspondence $\Gamma$ acts on cohomology as projector on 
      \[ \wedge^2 H^4_{tr}(Y)\ \ \subset\ H^8(Y\times Y)\ .\]
                      
  By hypothesis, $\dim H^4_{tr}(Y)=2$ and so $\wedge^2 H^4_{tr}(Y)$ is one--dimensional. Moreover, there is an inclusion
   \[ \wedge^2 H^4_{tr}(Y)\ \ \subset\ H^8(Y\times Y)\cap F^4\ ,\]  
   and so (since we assume the Hodge conjecture is true for $Y\times Y$) we have
   \[  \wedge^2 H^4_{tr}(Y) =\QQ [\pi_{}]\ \ \ \subset\ H^8(Y\times Y)  \]
   for some cycle $\pi\in A^4(Y\times Y)$.

Now, for brevity let us write $W:=\wedge^2 H^4_{tr}(Y)\subset H^8(Y\times Y)$.
The correspondence $\Gamma$ acts as projector on $W$. Hence, there is an inclusion
  \[ \Gamma\ \ \in  W\otimes W\ \ \ \subset H^{16}\bigl( (Y\times Y)\times (Y\times Y)\bigr)\ .\]
  Since $W$ is one--dimensional and generated by the cycle $\pi_{}$, this implies that
   \[ \Gamma = \gamma:=\lambda (p_{12})^\ast(\pi_{}) \cdot (p_{34})^\ast(\pi_{})\ \ \ \hbox{in}\ H^{16}\bigl( (Y\times Y)\times (Y\times Y)\bigr)\ ,\]
for some $\lambda\in\QQ$. Here, $p_{12}$ and $p_{34}$ are the projections from $Y^4$ to the first two (resp. last two) factors. In other words, we have
  \[ \Gamma-\gamma\ \ \ \in A^8_{hom}\bigl( (Y\times Y)\times (Y\times Y)\bigr)\ .\]
  But $Y\times Y$ has finite--dimensional motive, and so theorem \ref{nilp} ensures there exists $N\in\NN$ such that
   \[ (\Gamma-\gamma)^{\circ N}=0\ \ \ \hbox{in}\ A^8_{}\bigl( (Y\times Y)\times (Y\times Y)\bigr)\ .\]  
Developing this expression (and remembering that $\Gamma$ is idempotent), this means that
  \[ \Gamma=\Gamma^{\circ N}=Q_1 + \cdots + Q_r\ \ \ \hbox{in}\ A^8_{}\bigl( (Y\times Y)\times (Y\times Y)\bigr)\ ,\]  
  where each $Q_j$ is a composition of correspondences containing $\gamma$ at least once.
  The correspondence $\gamma$ is supported on $Z\times Z\subset Y^4$, where $Z\subset Y\times Y$ is a codimension $4$ subvariety (indeed, $Z$ is the support of the cycle $\pi_{}$). For this reason, $\gamma$ acts trivially on $2$--cycles, i.e.
    \[ \gamma_\ast =0\colon\ \ \ A^6(Y\times Y)\ \to\ A^6(Y\times Y)\ .\]
    (Indeed, the action factors over $A^6(\wt{Z})=0$, where $\wt{Z}\to Z$ is a resolution of singularities.)
   It follows that each $Q_j$, and hence also $\Gamma$, acts trivially on $2$--cycles:
   \[ \Gamma_\ast =0\colon\ \ \ A^6(Y\times Y)\ \to\ A^6(Y\times Y)\ .\]
This clinches the theorem: let $b, b^\prime\in A^3_{hom}(Y)$. Then
  \[ 2\Gamma_\ast(b\times b^\prime)= b\times b^\prime -b^\prime\times b =0\ \ \ \hbox{in}\ A^6(Y\times Y)\ .\]
\end{proof}

Finally, let us prove theorem \ref{main2}:
     
  \begin{proof}(of theorem \ref{main2}) We need to check that the two cubics of theorem \ref{main2} verify the two conditions of theorem \ref{main1}. Finite--dimensionality of the Fermat hypersurface is well--known; this follows from the Shioda inductive structure \cite{Shi}, \cite{KS}. Finite--dimensionality of the second cubic follows from \cite{excubic}, where it is proven that any smooth cubic fourfold of the form
  \[ X_0^3 + f(X_1,\ldots, X_5)=0 \]
  has finite--dimensional motive.
  
  The fact that $\dim N^2 H^4(Y)=21$ (i.e., $\dim H^4_{tr}(Y)=2$) is well--known for the Fermat cubic fourfold; Beauville \cite[Proposition 11]{Beau4} attributes this to Shioda. The fact that 
    \begin{equation}\label{2} \dim H^4_{tr}(Y)=2\end{equation}
  also for the second cubic is proven by Mongardi \cite{Mon}; this is an application of \cite[Proposition 1.2]{Mon} combined with the fact that the Fano variety $X=F(Y)$ admits an order $11$ symplectic automorphism.
(Alternatively, a different proof of equality (\ref{2}) for the second cubic can be found in \cite[Section 5.5.2]{BNS}.)

The Hodge conjecture is known for self--products of Fermat hypersurfaces of degree $\le 20$ \cite[Theorem IV]{Shi}. For the second cubic, let us write $V:=H^4_{tr}(Y)$. By hypothesis, the complexification $V_\C$ is such that
    \[  V_\C=H^{3,1}(Y)\oplus H^{1,3}(Y)\ .\]
%   Hence
%     \[ (V\otimes V)_\C = V_\C\otimes V_\C\ \ \subset\ H^{6,2}(Y\times Y)\oplus  H^{4,4}(Y\times Y)\oplus H^{2,6}(Y\times Y)\ .\]
 The complex vector space
   \[ (V\otimes V)_\C  \cap H^{4,4}(Y\times Y) = \bigl(  H^{3,1}(Y)\otimes H^{1,3}(Y)    \bigr) \oplus \bigl(  H^{1,3}(Y)\otimes H^{3,1}(Y)   \bigr) \]
  is two--dimensional, and so the $\QQ$--vector space
   \[ (V\otimes V)  \cap H^{4,4}(Y\times Y)   \]
   has dimension at most $2$. 
   %To prove the Hodge conjecture for $Y\times Y$, it thus suffices to find $2$ cycles on $Y\times Y$ that are linearly independent in $H^8(Y\times Y)$.
   Since the cubic $Y$ is a triple cover of $\PP^4$, there exists a non-symplectic automorphism $\sigma\in\aut(X)$ of order $3$.
   The cycles $\pi_{tr}$ and $\Gamma_\sigma\circ\pi_{tr}$ have cohomology class in $V\otimes V$. The first acts as the identity on $H^{3,1}(Y)$, while the second acts as multiplication by a primitive $3$rd root of unity. It follows that these two cycles are not proportional, and so they generate $ (V\otimes V)  \cap F^4$. This proves the Hodge conjecture for $Y\times Y$.   
  \end{proof}

\vskip1cm
\begin{nonumberingt} Thanks to "ik ben een kangaroe" Kai and thanks to "ik ben konijntje over" Len.
\end{nonumberingt}

\end{document}